\documentclass[a4paper,11pt]{article}
\usepackage{amsmath}
\usepackage{amsthm}
\usepackage{amssymb}
\usepackage[latin1]{inputenc}
\usepackage[english]{babel}
\usepackage{hyperref}

\theoremstyle{plain}
\newtheorem{thm}{Theorem}[section]
\newtheorem{prop}[thm]{Proposition}
\newtheorem{lem}[thm]{Lemma}
\newtheorem{cor}[thm]{Corollary}

\newtheorem*{claim}{Claim}
\theoremstyle{definition}

\theoremstyle{remark}

\DeclareMathOperator{\vol}{vol}

\DeclareMathOperator{\Fut}{Fut}
\DeclareMathOperator{\linspan}{span}

\begin{document}
\title{K-stability, Futaki invariants and cscK metrics on orbifold resolutions}
\date{\today}

\author{Claudio Arezzo\footnote{ICTP Trieste and Università di Parma, arezzo@ictp.it}, Alberto Della Vedova\footnote{Università di Milano-Bicocca, alberto.dellavedova@unimib.it}, Lorenzo Mazzieri\footnote{Università di Trento, lorenzo.mazzieri@unitn.it}}

\maketitle

\begin{abstract}
In this paper we compute the Futaki invariant of adiabatic K\"ahler classes on resolutions of K\"ahler orbifolds with isolated singularities.
Combined with previous existence results of extremal metrics by Arezzo-Lena-Mazzieri, this gives a number of new existence and non-existence results for cscK metrics.  
\end{abstract}

\section{Introduction}

In this paper we address the question of existence (and non-existence) of  constant scalar curvature K\"ahler metrics
(cscK from now on) in adiabatic K\"ahler classes on resolutions of compact cscK orbifolds with isolated singularities.

Form a purely conceptual point of view the basic existence result for extremal K\"ahler metrics proved in 
\cite{ArezzoLenaMazzieri2015} can be reinterpreted in the following form in the spirit of Szekelyhidi's work on blow ups of smooth points \cite{Gabor1,Gabor2}:

\begin{thm} 
Let $M$ be a K\"ahler orbifold of dimension $m$ with finite singular set $S \subset M$, and let $\pi : M' \to M$ be a resolution of singularities with local model $\pi_p : X_p \to \mathbf C^m/\Gamma_p$ at each $p \in S$.
Assume that $M$ admits a K\"ahler metric $\omega$ with constant scalar curvature and that each $X_p$ admits a scalar flat ALE K\"ahler metric $\eta_p$. 

 Then there exists $\varepsilon _0 >0$ such that for all $ \varepsilon < \varepsilon _0$, the following are equivalent
\begin{enumerate}
\item
$M'$ has a Kcsc metric in the class $\pi^{*}[\omega]+ \sum_{p\in S}\varepsilon [ \eta_p ]$
\item
$(M',   \pi^{*}[\omega ]+ \sum_{p\in S}\varepsilon [ \eta_p] )$ is $K$-stable.
\end{enumerate}
\end{thm}

$1 \implies 2$ is proved in \cite[Proposition 38]{Gabor2}, while $2\implies 1$  is a simple consequence of \cite[Theorem 1.1]{ArezzoLenaMazzieri2015}.

While this result settles the celebrated Tian-Yau-Donaldson Conjecture \cite{DT, Do} for these type of manifolds and classes, because of the known difficulty in checking $K$-stability for a polarized manifold, it remains of great interest to have some effective method to give some geometric conditions on $S$ which guarantee the existence of cscK metrics. This is the primary aim of this paper.

In \cite{ADVLM} and \cite{ArezzoLenaMazzieri2015} partial results have been obtained in this direction by performing
a careful analysis of the PDE, very much in the spirit of the analogue results of Arezzo-Pacard \cite{AP, AP2} for blow ups of smooth points. This approach produces a variety of sufficient conditions for the existence of a cscK metric, all of which follow from the more general results of the present work.

The main result of this paper, Theorem \ref{thm::mainthmFutaki}, is the computation of the Futaki invariant of adiabatic K\"ahler classes on resolution of singularities in terms of corresponding objects on the base orbifold and the geometry of $X_p$. In fact, the nonuniqueness of the resolution one decides to consider, prevents from using the algebraic techniques already employed in the analogue situation for the blow ups of smooth points
by many authors (Stoppa \cite{St}, \cite{DVZ}, Odaka \cite{Od} and Szekelyhidi \cite{Gabor1}).

What the PDE analysis showed is that a critical difference in the behaviour of this problem comes from the ADM mass of the local model of the resolution. While it has been longly known how to relate this number to the behaviour at infinity of ALE metrics, only very recently Hein-LeBrun \cite{HL} have discovered some very elegant interpretation of this quantity in purely cohomological terms for scalar flat metrics. Objects coming into the computation of the Futaki invariant are different from theirs, yet their work has been a key source of inspiration to bypass the problem of non-uniqueness of the resolution.

A number of consequences follows from our Theorem \ref{thm::mainthmFutaki} both getting a new proof of the results of \cite{ADVLM} and \cite{ArezzoLenaMazzieri2015}, but more importantly of new existence and various nonexistence results, which are discussed in details in Section~\ref{sec:csck} below. Using this approach, we can distinguish three different situations (it is worth recalling for the convenience of the reader that an ALE manifold is allowed to have zero ADM mass without being isometric to the flat Euclidean space, as pointed out in \cite{LeBrun1988}):

\begin{enumerate}
\item
For all $p\in S$, the local model $X_p$ has a scalar flat metric with zero ADM mass;
\item
There exist $p,q\in S$ such that the local model $X_p$ has a scalar flat metric with zero ADM mass, 
the local model $X_q$ has a scalar flat metric with non-zero ADM mass, and the adiabatic classes have the same scales of volumes of exceptional divisors;
\item
There exist $p,q\in S$ such that local model $X_p$ has a scalar flat metric with zero ADM mass, 
the local model $X_q$ has a scalar flat metric with non-zero ADM mass, and the adiabatic classes have different scales of volumes of exceptional divisors.
\end{enumerate} 

In each of these cases we give a sufficient condition which generalizes the ones found in
\cite{ADVLM} and \cite{ArezzoLenaMazzieri2015} in terms of the position of the singular points to be resolved, which we prove to be essentially also 
necessary in Theorem \ref{nonex}. This is done in Theorems  \ref{somezero}, \ref{allzero} and \ref{diffvolumes} respectively.

\section{The Futaki invariant of an orbifold resolution}

Let $(M,\omega)$ be a compact K\"ahler orbifold of complex dimension $m$ with finite singular set $S \subset M$.
This means that $M$ is a compact Hausdorff topological space endowed with a structure of $n$-dimensional complex manifold on the subset $M\setminus S$ such that for each singular point $p \in S$ there exist the following data:
\begin{itemize}
\item a neighborhood $U_p$ intersecting $S$ just at $p$,
\item a non-trivial finite subgroup $\Gamma_p \subset U(m)$,
\item a homeomorphism between $U_p$ and a ball $B(r)/\Gamma_p \subset \mathbf C^m/\Gamma_p$ which restricts to a biholomorphism between $U_p \setminus\{p\}$ and the punctured ball $B'(r)/\Gamma_p$
\end{itemize}
Moreover, $\omega$ restricts to a K\"ahler form of a genuine K\"ahler metric on $M \setminus S$, and the restriction of $\omega$ to $U_p$ lifts to a $\Gamma_p$-invariant K\"ahler form on the ball $B(r) \subset \mathbf C^m$.

The quotient $\mathbf C^m/\Gamma_p$ is called the local model for the singularity at $p$.
We stress that different singular points may have different local models. 
Finally note that, in principle, the radius $r$ may depend on $p$, but taking the minimum as $p$ varies on $S$ we can suppose that $r$ is indeed independent of the point.

\subsection{Resolution of orbifold singularities}\label{concentratedresolutions}

Let $p \in S$ be a singular point of $M$, and let 
\begin{equation*}
\pi_p : X_p \to \mathbf C^m/\Gamma_p
\end{equation*}
be a resolution of singularities of the local model at $p$.
In view of our applications, we will always assume that $X_p$ admits a K\"ahler metric.
By definition, $\pi_p$ is a proper birational morphism from a $m$-dimensional complex manifold $X_p$ to $\mathbf C^m/\Gamma_p$ which restricts to a biholomorphism on the complement of $\pi_p^{-1}(0)$.
It follows by definition that $\pi_p^{-1}(0)$ is a union of compact complex submanifolds of $X$.
The biholomorphism between $U_p \setminus \{p\}$ and the punctured ball $B'(r)/\Gamma_p$ existing by definition of complex orbifold, and the fact that $\pi_p$ is a biholomorphism on a neighborhood of $\partial B(r)/ \Gamma_p$ also allow to replace each neighborhood $U_p \subset M$ with the resolved ball $\pi_p^{-1} (B(r)/\Gamma_p)$ and obtain a complex manifold $M'$ and a resolution of singularities
\begin{equation*}
\pi : M' \to M.
\end{equation*}
This map collects all maps $\pi_p$ as $p$ varies in $S$.
More precisely, $\pi$ acts on $\pi_p^{-1}(B(r)/\Gamma_p)$ as the composition of $\pi_p$ together with the homeomorphism from $B(r)/\Gamma_p$ to $U_p$.  
Moreover $\pi$ is the identity on the complement of $\pi^{-1}(U)$, where $U$ is the union of all $U_p$ as $p$ varies in $S$.
In particular $\pi$ turns out to be a biholomorphism when restricted to the complement of $\pi^{-1}(S)$.

\subsection{A K\"ahler metric on $M'$}

In this subsection, we construct a K\"ahler metric on $M'$ which is, in some respect, a deformation of the K\"ahler metric $\omega$ on $M$. 

For any $p \in S$, let $\eta_p$ be a K\"ahler metric on the model resolution $X_p$ of the form
\begin{equation*}
\eta_p = \xi_p + dd^c \varphi_p,
\end{equation*}
where $\xi_p$ is a $(1,1)$-form supported in $\pi_p^{-1}(B(r)/\Gamma_p)$, and $\varphi_p$ is a smooth function.
Since we constructed $M'$ by replacing each singular ball $U_p$ with the resolved ball $\pi_p^{-1}(B(r)/\Gamma_p)$, we can think of each $\xi_p$ as a $(1,1)$-form on $M'$.
Thus, for all real $\varepsilon$, we can consider the following $(1,1)$-form on $M'$
\begin{equation*}
\omega_\varepsilon = \pi^* \omega + \varepsilon \sum_{p\in S} \xi_p.
\end{equation*}
Note that $\omega_\varepsilon$ defines a K\"ahler metric on $M'$ for all $\varepsilon>0$ sufficiently small.
This can be seen by considering the restriction of $\omega$ on $\pi^{-1}(U)$ and on $M' \setminus \pi^{-1}(U)$.
The latter is positive since there the map $\pi$ restricts to a biholomorphism and $\xi$ vanishes.
Thus it remains to check that for $\varepsilon$ sufficiently small $\omega_\varepsilon$ is positive around the resolution $\pi^{-1}(U_p)$ of any singular point $p \in S$.
Over $\pi^{-1}(U_p)$ the form $\omega_\varepsilon$ restricts to $\pi_p^*\omega + \varepsilon \xi_p$ (here, for ease of notation, we wrote $\omega$ instead of the pullback of $\omega|_{U_p}$ to the ball $B(r)/\Gamma_p$).
Since $\omega$ comes from a $\Gamma_p$-invariant form on $B(r)$, we can suppose $\omega = dd^c h_p$ for some smooth function $h_p$ on the ball $B(r)/\Gamma_p$.
As a consequence, on $\pi^{-1}(U_p)$ we have
\begin{equation*}
\omega_\varepsilon = dd^c (\pi_p^* h_p - \varepsilon \varphi_p) + \varepsilon \eta_p.
\end{equation*}
Note that $\eta_p$ is positive in any neighborhood of the exceptional set $\pi_p^{-1}(p)$.
On the other hand, for $\varepsilon$ sufficiently small the function $\pi_p^* h_p - \varepsilon \varphi_p$ is plurisubharmonic on the complement of the exceptional set.
This shows that $\omega_\varepsilon$ is positive on any $\pi^{-1}(U_p)$ provided $\varepsilon$ is sufficiently small, as claimed.

\subsection{Pushing down vector fields}

In this section we show that any holomorphic vector field on $M'$ induces a holomorphic vector field on $M$.
If the vector field on $M'$ is also Hamiltonian, the induced vector field on $M$ is Hamiltonian too. 

\begin{lem}\label{holomorphicpushdown}
Any holomorphic vector field $V$ on $M'$ descends to a holomorphic vector field $\pi_*V$ on $M$ which vanishes at all points of $S$.
\end{lem}
\begin{proof}
Since $\pi$ is a biholomorphism on the complement of $\pi^{-1}(S)$, pushing down the restriction of $V$ to that set defines a vector field $V'$ on $M \setminus S$.
Given $p \in S$, the restriction of $V'$ to $U_p \setminus\{p\}$ lifts to a $\Gamma_p$-invariant vector field on the punctured ball $B'(r)$ of $\mathbf C^m$. 
By Hartog's theorem such a vector field extends to a holomorphic vector field on the whole ball $B(r)$.
Of course such a vector field is $\Gamma_p$-invariant, and so it gives a holomorphic vector field on $U_p$ which is equal to $V'$ on $U_p \setminus\{p\}$. 
Therefore one ends up with a holomorphic vector field $\pi_*V$ on $M$.
 
It remains to show that $\pi_*V$ vanishes at any $p \in S$.
To this end, note that the fiber $\pi^{-1}(p)$ is a union of compact complex submanifolds.
Therefore $V$ must be tangent to it, and consequently $\pi_*V$ must tend to zero as approaching to $p$.
By continuity we can then conclude that $\pi_*V$ actually vanishes at $p$.
\end{proof}

\begin{lem}\label{hamiltonianpushdown}
If $V$ is a holomorphic vector field on $M'$ and is Hamiltonian with respect to $\omega_\varepsilon$, then $\pi_*V$ is Hamiltonian with respect to $\omega$ on $M$.
Moreover, if $u_\varepsilon$ and $u$ are Hamiltonian potentials for $V$ and $\pi_*V$ respectively, then one has
\begin{equation}\label{expansionpotential}
u_\varepsilon = \pi^*u + \varepsilon \sum_{p \in S} u_p + c(\varepsilon),
\end{equation}
where $u_p$ is a smooth function supported in $\pi^{-1}(U_p)$ satisfying $ d u_p = i_V \xi_p $, and $c(\varepsilon)$ is a constant.
\end{lem}
\begin{proof}
On the complement of $\pi^{-1}(U)$ the K\"ahler form $\omega_\varepsilon$ is equal to $\omega$ and the vector field $V$ is equal to $\pi_*V$.
Therefore the Hamiltonian potential $u_\varepsilon$ of $V$ restricts to a function on $M' \setminus \pi^{-1}(U)$, say $u$, which does not depend on $\varepsilon$ and is a Hamiltonian potential for $\pi_*V$ with respect to $\omega$ on $M \setminus U$.

Now let $p \in S$.
Identifying $\pi^{-1}(U_p)$ with the resolved ball $\pi_p^{-1} (B(r)/\Gamma_p)$ in the local model $X_p$,
one can write
\begin{equation*}
\omega_\varepsilon = \pi_p^* dd^c h_p + \varepsilon \xi_p,
\end{equation*} 
where $h_p$ is a K\"ahler potential for $\omega$ in the ball $B(r)/\Gamma_p$.
Since, by hypothesis, $V$ is Hamiltonian with potential $u_\varepsilon$, one has 
\begin{equation*}
d u_\varepsilon = i_V dd^c \pi_p^*h_p + \varepsilon i_V \xi_p.
\end{equation*}
On the other hand, $V$ is holomorphic, therefore Cartan formula yields
\begin{equation*}
d u_\varepsilon = d^c V(\pi_p^*h_p) +\frac{1}{4} d JV (\pi_p^*h_p) + \varepsilon i_V \xi_p,
\end{equation*}
whence it follows that
\begin{equation*}
i_V \xi_p = d u_p - \frac{1}{\varepsilon}d^c V(\pi_p^*h_p),
\end{equation*}
for some smooth function $u_p$.
%
%
Note that the last summand does not depend on $\xi_p$, but it is forced to vanish where $\xi_p$ does.
Therefore, by arbitrariness of the metric $\eta_p$ we started with, it must vanish everywhere.
As a consequence the support of $d u_p$ is contained in the support of $\xi_p$.
In particular, up to adding a suitable constant, we can suppose that $u_p$ is supported in $\pi^{-1}(U_p)$.
Thus $u_p$ has all the properties stated above.
Moreover we proved that on $\pi^{-1}(U_p)$ it holds
\begin{equation*}
u_\varepsilon = \frac{1}{4} JV (\pi_p^*h_p) + \varepsilon u_p + c(p,\varepsilon),
\end{equation*}
where $c(p,\varepsilon)$ is a constant.

Since the support of $\xi_p$ is compactly contained in the resolved ball $\pi_p^{-1}(B(r)/\Gamma_p)$, on a neighborhood of the boundary $\pi^{-1}(\partial U_p)$ it holds
\begin{equation*}
u_\varepsilon = \pi_p^* \left( \frac{1}{4} (J\pi_*V) (h_p) + c(p,\varepsilon) \right).
\end{equation*}
Therefore $u$ extends to $\frac{1}{4} (J\pi_*V) (h_p) + c(p,\varepsilon)$ on $U_p$.
Finally note that it holds $du=i_{\pi_*V}\omega$, that is $u$ is a Hamiltonian potential for $\pi_*V$.
Since Hamiltonian potentials are defined just up to an additive constant, equation \eqref{expansionpotential} then follows.
\end{proof}

\subsection{The Futaki invariant}

Given a holomorphic vector field $V$ on the resolution $M'$, and supposing that $V$ is Hamiltonian with respect to $\omega_\varepsilon$ with potential $u_\varepsilon$, one can form the Futaki invariant
\begin{equation*}
\Fut(V,\omega_\varepsilon)
= \int_{M'} (u_\varepsilon - \underline{u_\varepsilon}) \frac{\rho_\varepsilon \wedge \omega_\varepsilon^{m-1}}{(m-1)!},
\end{equation*}
where $\rho_\varepsilon$ is the Ricci form of $\omega_\varepsilon$, and $\underline{u_\varepsilon} = \int u_\varepsilon \omega_\varepsilon^m / \int \omega_\varepsilon^m$ is the mean value of $u_\varepsilon$ with respect to $\omega_\varepsilon$.

On the other hand, thanks to Lemmata \ref{holomorphicpushdown} and \ref{hamiltonianpushdown}, $V$ descends to a holomorphic vector field $\pi_*V$ on $M$ wich is Hamiltonian with respect to $\omega$ with potential, say, $u$.
Thus one can also consider the Futaki invariant
\begin{equation*}
\Fut(\pi_*V,\omega)
= \int_M (u - \underline{u}) \frac{\rho \wedge \omega^{m-1}}{(m-1)!},
\end{equation*}
where $\rho$ is the Ricci form of $\omega$, and $\underline{u} = \int u \,\omega^m / \int \omega^m$.
The Futaki invariants $\Fut(V,\omega_\varepsilon)$ and $\Fut(\pi_*V,\omega)$ are relate by the following
\begin{thm}\label{thm::mainthmFutaki}
As $\varepsilon \to 0$ one has
\begin{multline}\label{expansionfutaki}
\Fut(V,\omega_\varepsilon)
= \Fut(\pi_*V,\omega)
+ \varepsilon^{m-1} \sum_{p \in S} (u(p) - \underline{u}) \int_{X_p} \frac{\rho_p \wedge \xi_p^{m-1}}{(m-1)!} \\
- \varepsilon^m \sum_{p \in S} \left(\underline{s} (u(p) - \underline{u}) + \Delta u(p)\right) \int_{X_p} \frac{\xi_p^m}{m!} + O(\varepsilon^{m+1}).
\end{multline}
where $\rho_p$ is the Ricci form of the chosen ALE K\"ahler metric $\eta_p$ on the model resolution $X_p$, and $\underline{s} = m \int \rho \wedge \omega^{m-1} / \int \omega^m$ is the mean scalar curvature of $\omega$.
\end{thm}

\begin{proof}
The Futaki invariant of the vector field $V$ can be written as
\begin{equation}\label{FutakiVomegaepsilon}
\Fut(V,\omega_\varepsilon)
= \int_{M'} (\rho_\varepsilon - \Delta_\varepsilon u_\varepsilon) \wedge \frac{(\omega_\varepsilon + u_\varepsilon)^m}{m!}
- \underline{u_\varepsilon} \int_{M'} \rho_\varepsilon \wedge\frac{\omega_\varepsilon^{m-1}}{(m-1)!},
\end{equation}
where $\Delta_\varepsilon$ denotes the Laplacian of the K\"ahler metric on $M'$ associated to $\omega_\varepsilon$.
The first integral and the average of $u_\varepsilon$ perhaps could be calculated using equivariant cohomology theory.
However one can avoid that theory and prove the statement by means of more elementary arguments.
In order to understand the formula above, note that $\rho_\varepsilon - \Delta_\varepsilon u_\varepsilon$ and $\omega_\varepsilon + u_\varepsilon$ are non-homogeneous differential forms on $M'$.
Their wedge product is the usual one, and so the integrand is a sum of even degree differential forms.
The integral of such a form is, by definition, the integral of its $2m$-degree component.
Now consider the differential operator
\begin{equation*}
d_V = d - i_V
\end{equation*}
acting on differential froms on $M'$.
The fact that $u_\varepsilon$ is a Hamiltonian potential for $V$ with respect to $\omega_\varepsilon$ can be stated as
\begin{equation*}
d_V (\omega_\varepsilon + u_\varepsilon) = 0.
\end{equation*}
In other words, $\omega_\varepsilon + u_\varepsilon \in \Omega^*(M)$ is a $d_V$-closed differential form. 
After expanding
\begin{equation}\label{expomegaepsilon+uepsilon}
\omega_\varepsilon + u_\varepsilon = \pi^*(\omega+u) + \varepsilon \sum_{p\in S} \xi_p+u_p,
\end{equation}
one sees that $\pi^*(\omega+u)$ and $\xi_p+u_p$ are $d_V$-closed as well.
Finally, by a standard local calculation one can check that
\begin{equation*}
d_V (\rho_\varepsilon - \Delta_\varepsilon u_\varepsilon) = 0.
\end{equation*}

Now pick a singular point $p \in S$ and let $\pi^{-1}(p)$ be the exceptional divisor over $p$.
The proof of the statement of the Theorem rests essentially on the following
\begin{claim}
Any $d_V$-closed differential form $\alpha \in \Omega^*(M')$ which restricts to the zero form on the exceptional divisor $\pi^{-1}(p)$ satisfies
$\int_{M'} \alpha \wedge (\xi_p + u_p) = 0.$
\end{claim} 

In order to prove the claim note that by Poincaré-Lelong equation one can find a path $F_t$ of smooth functions on $M'$, with $t>0$, such that $\xi_p + dd^c F_t$ weekly converges to a current supported on $\pi^{-1}(p)$ as $t \to 0$.
A moment's thought should show that $\xi_p + u_p + d_Vd^c F_t$ also converges to the same current.
On the other hand, note that by Stokes' Theorem the integral of any $d_V$-closed form on $M'$ vanishes.
As a consequence one has
\begin{equation*}
\int_{M'} \alpha \wedge (\xi_p + u_p) = \int_{M'} \alpha \wedge (\xi_p + u_p + d_V d^c F_t) \to 0
\end{equation*}
as $t \to 0$ thanks to the hypothesis that $\alpha$ restricts to the zero form on $\pi^{-1}(p)$.

\bigskip
Now we can proceed with the proof of the Theorem.
We need to expand \eqref{FutakiVomegaepsilon} in powers of $\varepsilon$.
First of all consider the average $\underline{u_\varepsilon} = \int u_\varepsilon \omega_\varepsilon^m / \int \omega_\varepsilon^m$.
Note that
\begin{equation}\label{integraluvarepsilon=equiv}
(m+1) \int_{M'} u_\varepsilon \omega_\varepsilon^m = \int_{M'} (\omega_\varepsilon + u_\varepsilon)^{m+1}. 
\end{equation}
Substituting \eqref{expomegaepsilon+uepsilon} yields
\begin{multline}\label{equivintegraluvarepsilon}
\int_{M'} (\omega_\varepsilon + u_\varepsilon)^{m+1}
= \int_M (\omega+u)^{m+1} + \sum_{\ell=1}^m \varepsilon^\ell \binom{m+1}{\ell} \sum_{p\in S} \int_{M'} \pi^*(\omega+u)^{m+1-\ell} \wedge (\xi_p + u_p)^\ell \\
+ \varepsilon^{m+1} \sum_{p\in S} \int_{M'} (\xi_p + u_p)^{m+1}.
\end{multline}
Focus on the middle summands of the right hand side.
Note that for all $1 \leq \ell \leq m$ and $p \in S$ the differential form
\begin{equation*}
\alpha = \pi^*(\omega+u)^{m+1-\ell} \wedge (\xi_p + u_p)^{\ell-1} - u(p)^{m+1-\ell} (\xi_p + u_p)^{\ell-1}
\end{equation*}
satisfies the hypotheses of the Claim, whence it follows
\begin{equation*}
\int_{M'} \pi^*(\omega+u)^{m+1-\ell} \wedge (\xi_p + u_p)^\ell = u(p)^{m+1-\ell} \int_{M'} (\xi_p + u_p)^\ell.
\end{equation*}
Since $\ell$ runs from $1$ to $m$, the right hand side integral vanishes unless $\ell=m$, in which case it reduces to the integral of $\xi_p^m$.
Therefore, substituting in \eqref{equivintegraluvarepsilon} and \eqref{integraluvarepsilon=equiv} we get the expansion
\begin{equation}\label{asymexpintuvarepsilon}
\int_{M'} u_\varepsilon \omega_\varepsilon^m
= \int_M u \, \omega^m + \varepsilon^m \sum_{p\in S} u(p) \int_{M'} \xi_p^m
+ O(\varepsilon^{m+1}).
\end{equation}

In order to get the expansion of $\underline {u_\varepsilon}$ we need to divide the expression above by the total volume of $\omega_\varepsilon$.
Recalling that $\omega_\varepsilon = \pi^*\omega + \varepsilon \sum \xi_p$, one has
\begin{equation*}
\int_{M'} \omega_\varepsilon^m
= \int_M \omega^m + \sum_{\ell=1}^m \varepsilon^\ell \binom{m}{\ell} \sum_{p\in S} \int_{M'} \pi^*\omega^{m-\ell} \wedge \xi_p^\ell.
\end{equation*}
Since $\xi_p$ is supported on $\pi^{-1}(U_p)$ and $\pi^*\omega$ is exact on there, Stokes' Theorem yields\begin{equation}\label{asymexpvolume}
\int_{M'} \omega_\varepsilon^m
= \int_M \omega^m + \varepsilon^m \sum_{p\in S} \int_{X_p} \xi_p^m.
\end{equation}
Note that we replaced the integral of $\xi_p$ over $M'$ with the integral over the model resolution $X_p$.
This is possible since the support of $\xi_p$ is contained in $\pi^{-1}(U_p)$, which in turn we have identified with a neighborhood of the exceptional divisor of $X_p$.
Dividing \eqref{asymexpintuvarepsilon} by \eqref{asymexpvolume} finally gives
\begin{equation}\label{finalexpansionuvarepsilon}
\underline{u_\varepsilon}
= \underline{u} + \varepsilon^m \sum_{p\in S} (u(p) - \underline{u}) \frac{\int_{X_p} \xi_p^m}{\int_M \omega^m}
+ O(\varepsilon^{m+1}).
\end{equation}

Now we pass to consider the total scalar curvature of $\omega_\varepsilon$.
Arguing as above, after susbstituting $\omega_\varepsilon = \pi^*\omega + \varepsilon \sum \xi_p$ and applying Stokes' Theorem, one gets
\begin{equation}\label{totalscalarcurvaturesplit}
\int_{M'} \rho_\varepsilon \wedge \omega_\varepsilon^{m-1}
= \int_{M'} \rho_\varepsilon \wedge \pi^*\omega^{m-1}
+ \varepsilon^{m-1} \sum_{p \in S} \int_{M'} \rho_\varepsilon \wedge \xi_p^{m-1}.
\end{equation}
Consider the two summands separately.
Adding and subctracting $\pi^* \rho$ in the first summand yields
\begin{equation*}
\int_{M'} \rho_\varepsilon \wedge \pi^*\omega^{m-1}
= \int_{M'} \pi^*\rho \wedge \pi^*\omega^{m-1} + \int_{M'} (\rho_\varepsilon - \pi^* \rho) \wedge \pi^*\omega^{m-1}.
\end{equation*}
The first summand reduces to the integral of $\rho \wedge \omega^{m-1}$ over $M$, and the second summand vanishes once again by Stokes' Theorem.
Indeed $\omega_\varepsilon = \omega$ on the complement of $U$.
As a consequence $\rho_\varepsilon - \pi^*\rho$ vanishes on that set.
On the other hand $\pi^* \omega$ is exact on any connected component of $U$.
Therefore
\begin{equation}\label{firstsummandtotalscalarcurvature}
\int_{M'} \rho_\varepsilon \wedge \pi^*\omega^{m-1}
= \int_M \rho \wedge \omega^{m-1}.
\end{equation}
Now consider the second summand of the right hand side of \eqref{totalscalarcurvaturesplit}.
Since the support of $\xi_p$ is contained in $\pi^{-1}(U_p)$, which we identified with a neighborhood of the exceptional divisor of the model resolution $X_p$, we can consider the Ricci form $\rho_p$ of the chosen ALE K\"ahler metric $\eta_p$ on $X_p$ and thought of $\rho_p \wedge \xi_p^{m-1}$ as a differential form on $M'$.
After noting that $\rho_\varepsilon-\rho_p = dd^c \log(\eta_p^m/\omega_\varepsilon^m)$, Stokes' Theorem yields
\begin{equation*}
\int_{M'} (\rho_\varepsilon-\rho_p) \wedge \xi_p^{m-1} = 0.
\end{equation*}
On the other hand, we can also consider $\rho_p \wedge \xi_p^{m-1}$ as a differential form on the local resolution $X_p$.
As a consequence we can rewrite equation above in the form
\begin{equation}\label{intrhovarepsilonxipn-1}
\int_{M'} \rho_\varepsilon \wedge \xi_p^{m-1}
= \int_{X_p} \rho_p \wedge \xi_p^{m-1}.
\end{equation}
Therefore substituting \eqref{intrhovarepsilonxipn-1} together with \eqref{firstsummandtotalscalarcurvature} into \eqref{totalscalarcurvaturesplit} yields
\begin{equation}\label{totalscalarcurvatureomegavarepsilon}
\int_{M'} \rho_\varepsilon \wedge \omega_\varepsilon^{m-1}
= \int_M \rho \wedge \omega^{m-1} + \varepsilon^{m-1} \sum_{p\in S} \int_{X_p} \rho_p \wedge \xi_p^{m-1}.
\end{equation}
Thanks to \eqref{finalexpansionuvarepsilon} and \eqref{totalscalarcurvatureomegavarepsilon},
the second summand of the right hand side of \eqref{FutakiVomegaepsilon} expands as
\begin{multline}\label{FutakiVomegaepsilonsecondsummand}
\underline{u_\varepsilon} \int_{M'} \rho_\varepsilon \wedge\frac{\omega_\varepsilon^{m-1}}{(m-1)!}
= \underline{u} \int_M \rho \wedge\frac{\omega^{m-1}}{(m-1)!}
+ \varepsilon^{m-1} \sum_{p\in S} \underline{u} \int_{X_p} \frac{\rho_p \wedge \xi_p^{m-1}}{(m-1)!} \\
+ \varepsilon^m \sum_{p\in S} \underline{s} (u(p) - \underline{u}) \int_{X_p} \frac{\xi_p^m}{m!}
+ O(\varepsilon^{m+1}).
\end{multline}

Finally it remains to consider the first summand of \eqref{FutakiVomegaepsilon}.
Substituting \eqref{expomegaepsilon+uepsilon} yields
\begin{multline}\label{futakifirstexpansion}
\int_{M'} (\rho_\varepsilon - \Delta_\varepsilon u_\varepsilon) \wedge \frac{(\omega_\varepsilon + u_\varepsilon)^m}{m!}
= \int_{M'} (\rho_\varepsilon - \Delta_\varepsilon u_\varepsilon) \wedge \frac{\pi^*(\omega + u)^m}{m!} \\
+ \sum_{\ell=1}^m \varepsilon^\ell \sum_{p \in S} \int_{M'} (\rho_\varepsilon - \Delta_\varepsilon u_\varepsilon) \wedge \frac{\pi^*(\omega + u)^{m-\ell}}{(m-\ell)!} \wedge \frac{\left(\xi_p + u_p\right)^\ell}{\ell!}.
\end{multline}
Focus on the summands of the second line.
Fix $p \in S$ and suppose $0 < \ell <m$, so that the differential form
\begin{equation*}
\alpha = (\rho_\varepsilon - \Delta_\varepsilon u_\varepsilon) \wedge \frac{\pi^*(\omega + u)^{m-\ell} - u(p)^{m-\ell}}{(m-\ell)!} \wedge \frac{\left(\xi_p + u_p\right)^{\ell-1}}{\ell!}
\end{equation*}
satisfies the hypotheses of the Claim above, whence it follows that
\begin{equation*}
\int_{M'} (\rho_\varepsilon - \Delta_\varepsilon u_\varepsilon) \wedge \frac{\pi^*(\omega + u)^{m-\ell}}{(m-\ell)!} \wedge \frac{\left(\xi_p + u_p\right)^\ell}{\ell!}
= \frac{u(p)^{m-\ell}}{(n-\ell)!} \int_{M'} (\rho_\varepsilon - \Delta_\varepsilon u_\varepsilon) \wedge \frac{\left(\xi_p + u_p\right)^\ell}{\ell!}.
\end{equation*}
If $\ell<m-1$, the integrand differential form of the right hand side has no component of degree $2m$, therefore the integral is zero.
On the other hand, for $\ell = m-1$, equation above together with \eqref{intrhovarepsilonxipn-1} give
\begin{equation*}
\int_{M'} (\rho_\varepsilon - \Delta_\varepsilon u_\varepsilon) \wedge \pi^*(\omega + u) \wedge \frac{\left(\xi_p + u_p\right)^{m-1}}{(m-1)!}
= u(p) \int_{X_p} \frac{\rho_p \wedge \xi_p^{m-1}}{(m-1)!}.
\end{equation*}
By discussion above, \eqref{futakifirstexpansion} reduces to
\begin{multline}\label{futakisecondexpansion}
\int_{M'} (\rho_\varepsilon - \Delta_\varepsilon u_\varepsilon) \wedge \frac{(\omega_\varepsilon + u_\varepsilon)^m}{m!}
= \int_{M'} (\rho_\varepsilon - \Delta_\varepsilon u_\varepsilon) \wedge \frac{\pi^*(\omega + u)^m}{m!} \\
+ \varepsilon^{m-1} \sum_{p \in S} u(p) \int_{X_p} \frac{\rho_p \wedge \xi_p^{m-1}}{(m-1)!}
+ \varepsilon^m \sum_{p \in S} \int_{M'} (\rho_\varepsilon - \Delta_\varepsilon u_\varepsilon) \wedge \frac{\left(\xi_p + u_p\right)^m}{m!}.
\end{multline}
In order to treat the first summand of the right hand side, note that the difference $\rho_\varepsilon - \Delta_\varepsilon u_\varepsilon - \pi^*(\rho - \Delta u)$ is a $d_V$-closed differential form on $M'$ which is compactly supported in the union of all $\pi^{-1}(U_p)$ as $p$ varies in $S$.
The proof of the Claim above works also replacing the form $\xi_p + u_p$ with $\rho_\varepsilon - \Delta_\varepsilon u_\varepsilon - \pi^*(\rho - \Delta u)$.
As a consequence, since $\pi^*(\omega+u)^m$ restricts to a constant function on any exceptional divisor, one then has
\begin{equation*}
\int_{M'} (\rho_\varepsilon - \Delta_\varepsilon u_\varepsilon) \wedge \frac{\pi^*(\omega + u)^m}{m!}
= \int_{M'} \pi^*(\rho - \Delta u) \wedge \frac{\pi^*(\omega + u)^m}{m!}.
\end{equation*}
Therefore \eqref{futakisecondexpansion} simplifies to
\begin{multline}\label{futakithirdexpansion}
\int_{M'} (\rho_\varepsilon - \Delta_\varepsilon u_\varepsilon) \wedge \frac{(\omega_\varepsilon + u_\varepsilon)^m}{m!}
= \int_M (\rho - \Delta u) \wedge \frac{(\omega + u)^m}{m!}
+ \varepsilon^{m-1} \sum_{p \in S} u(p) \int_{X_p} \frac{\rho_p \wedge \xi_p^{m-1}}{(m-1)!} \\
+ \varepsilon^m \sum_{p \in S} \int_{M'} (\rho_\varepsilon - \Delta_\varepsilon u_\varepsilon) \wedge \frac{\left(\xi_p + u_p\right)^m}{m!}.
\end{multline}
Finally consider the last summand, which can be rewritten in the form
\begin{multline}\label{equivrhoepsilonwedgeequivxi}
\int_{M'} (\rho_\varepsilon - \Delta_\varepsilon u_\varepsilon) \wedge \frac{\left(\xi_p + u_p\right)^m}{m!}
= \int_{M'} \pi^*(\rho - \Delta u) \wedge \frac{\left(\xi_p + u_p\right)^m}{m!} \\
+ \int_{M'} \left(\rho_\varepsilon - \Delta_\varepsilon u_\varepsilon - \pi^*(\rho - \Delta u)\right) \wedge \frac{\left(\xi_p + u_p\right)^m}{m!}.
\end{multline}
The first summand of the right hand side can be trated once again by the Claim above and the Stokes' Theorem.
Indeed the differential form $\alpha = \left(\pi^*(\rho - \Delta u) + \Delta u(p)\right)  \wedge \left(\xi_p + u_p\right)^{m-1}$ on $M'$ satisfies the hypothesis of the Claim, whence arguing as above yields
\begin{equation*}
\int_{M'} \pi^*(\rho - \Delta u) \wedge \frac{\left(\xi_p + u_p\right)^m}{m!}
= -\Delta u (p) \int_{X_p} \frac{\xi_p^m}{m!}.
\end{equation*}
On the other hand, the second summand of \eqref{equivrhoepsilonwedgeequivxi} is $O(\varepsilon)$, as follows by \eqref{expomegaepsilon+uepsilon}.
As a consequence, \eqref{equivrhoepsilonwedgeequivxi} reduces to $
\int_{M'} (\rho_\varepsilon - \Delta_\varepsilon u_\varepsilon) \wedge \left(\xi_p + u_p\right)^m
= -\Delta u (p) \int_{X_p} \xi_p^m + O(\varepsilon)$.
Substituting this into \eqref{futakithirdexpansion} yields
\begin{multline}\label{futakifourthexpansion}
\int_{M'} (\rho_\varepsilon - \Delta_\varepsilon u_\varepsilon) \wedge \frac{(\omega_\varepsilon + u_\varepsilon)^m}{m!}
= \int_M (\rho - \Delta u) \wedge \frac{(\omega + u)^m}{m!} \\
+ \varepsilon^{m-1} \sum_{p \in S} u(p) \int_{X_p} \frac{\rho_p \wedge \xi_p^{m-1}}{(m-1)!}
- \varepsilon^m \sum_{p \in S} \Delta u (p) \int_{X_p} \frac{\xi_p^m}{m!} + O(\varepsilon^{m+1}).
\end{multline}
Finally the thesis follows by plugging this and \eqref{FutakiVomegaepsilonsecondsummand} into \eqref{FutakiVomegaepsilon}.
\end{proof}

\section{ALE resolutions}\label{sec:ALEres}

In this section we introduce K\"ahler metrics on local models having behavior at infinity suitable for applications in next section.

Let $\Gamma$ be a finite subgroup of the unitary group $U(m)$ and suppose that $\Gamma$ acts freely on the complement of $0 \in \mathbf C^m$.
The quotient $(\mathbf C^m \setminus \{0\})/\Gamma$ is therefore a complex manifold and the Euclidean metric on $\mathbf C^m$ descend to a K\"ahler metric $\eta_0$ on the quotient.
Such a K\"ahler metric serves as a model at infinity for ALE resolutions.

By ALE resolution (of $\mathbf C^m/\Gamma$) we mean a non-compact complex manifold $X$ equipped with a complete K\"ahler metric $\eta$ satisfying the following requirements:
\begin{itemize}
\item There exists a finite subgroup $\Gamma \subset U(m)$ acting freely on $\mathbf C^m \setminus \{0\}$ and a proper birational morphism $\pi : X \to \mathbf C^m/\Gamma$ which restricts to a biholomorphism of $X \setminus \pi^{-1}(0)$ onto $(\mathbf C^m \setminus \{0\})/ \Gamma$.
\item The metric $\eta$ approximates smoothly the model metric $\eta_0$ at infinity.
More specifically, for all integers $k \geq 0$ one has
\begin{equation}\label{vanishderiv}
\nabla^k (\pi_*\eta - \eta_0) = O(|z|^{2-2m-k}) \quad \mbox{ as } |z| \to \infty,
\end{equation}
where $\nabla$ denotes the Euclidean connection.
\end{itemize}

In particular, an ALE resolution turns out to be an ALE K\"ahler manifold subject to a couple of more restricting requirements.
Firstly, here we assume that $\pi$ is a biholomorphism at all smooth points of the quotient $\mathbf C^m / \Gamma$ (hence the name \emph{resolution}), wheras an ALE K\"ahler manifold (with one end) is just required to contain a compact subset $K$ whose complement is biholomorphic to the complement of a ball centered at the origin in $\mathbf C^m / \Gamma$. 
Secondly, ALE K\"ahler metrics are commonly allowed to have quite permissive fall-off order at infinity.
However, in second point above the order $2-2m$ is chosen in accordance with the well-known decay of scalar flat K\"ahler metrics on $\mathbf C^m / \Gamma$ \cite[Lemma 7.2]{AP}, being metrics of that kind our main interest for applications.  

On ALE resolutions one can develop Hodge theory as for compact K\"ahler manifolds \cite[pp. 182-186]{Joyce}.
The upshot is that any cohomology class in $H^{1,1}(X)$ can be represented by a closed compactly supported $(1,1)$-form on $X$ \cite[Theorem 8.4.3]{Joyce}.
Moreover one can suppose that 
\begin{equation}\label{eta=epsilon+ddcphi}
\eta = \xi + dd^c \varphi
\end{equation}
for some $(1,1)$-from $\xi$ compactly supported around the exceptional locus $\pi^{-1}(0)$ and some smooth real function $\varphi$ on $X$ \cite[Theorem 8.4.4]{Joyce}.
Clearly $\xi$ and $\varphi$ are not uniquely defined.
In particular $\varphi$ can be added by a function in the kernel of $dd^c$ operator.
However we rule out this indeterminacy by requiring that
\begin{equation*}
\varphi - |z|^2/4 = O(|z|^{4-2m}) 
\end{equation*} 
for large $z$.
Therefore by \eqref{vanishderiv} the derivative $\nabla^k(\varphi - |z|^2/4)$ must be $O(|z|^{4-2m-k})$ for all $k \geq 0$. 

Finally, here we recall an elementary result that will be useful in the following.
\begin{lem}\label{lemareasfera}
Let $B(R) \subset \mathbf C^m$ be the ball centered at zero of radius $R$. One has 
\begin{equation*} 
\int_{\partial B(R)} d^c|z|^2 \wedge (dd^c|z|^2)^{m-1} = (4\pi)^m R^{2m}. 
\end{equation*}
\end{lem}
\begin{proof}
First of all note that $dd^c|z|^2 = 2i \sum_{j=1}^m dz_j \wedge d\bar z_j$, whence it follows that $(dd^c|z|^2)^m$ is $4^mm!$ times the Euclidean volume form $\Omega_E$.
Therefore, by Stokes theorem one has
\begin{equation*} 
\int_{\partial B(R)} d^c|z|^2 \wedge (dd^c|z|^2)^{m-1} = 4^m m! \int_{B(R)} \Omega_E, 
\end{equation*}
whence the thesis follows by $ \int_{B(R)} \Omega_E = \vol(S^{2m-1})\int_0^R t^{2m-1}dt = \frac{2\pi^m}{\Gamma(m)} \frac{R^{2m}}{2m}$.
\end{proof}

\subsection{Asymptotic formulae for volume and total scalar curvature}

Let $m>1$ be an integer and consider the real function $f$ on $(0,+\infty)$ defined by 
\begin{equation}\label{deff}
f(t) = \frac{1}{4}t + e \frac{1-t^{2-m}}{2-m} + ct^{1-m}
\end{equation}
for some real constants $e,c$.
In this section we consider an $m$-dimensional ALE resolution whose K\"ahler metric is of the form $\eta = \xi + dd^c \varphi$ with $\xi$ compactly supported around $\pi^{-1}(0)$ and $\varphi$ satisfying
\begin{equation}\label{derivphiasym}
\nabla^k \varphi = \nabla^k f(|z|^2) + O(|z|^{-2m-k}) \quad \mbox{ as } |z| \to +\infty
\end{equation}
for all integer $k \geq 0$.
This ensures that $\eta$ satisfies the fall-off requirement \eqref{vanishderiv} of the third point of definition of ALE resolution.
Moreover, note that the second summand of $f$ is chosen so that $f$ depends smoothly (in fact analytically) on the dimension $m$ and for $m=2$ one has $f(t) = t/4 - e\log(t) + ct^{-1}$. 

We shall compute, at least up to some controlled error, the volume and the total scalar curvature of some subsets of $X$ with respect to $\eta$.
More specifically we shall give an asymptotic formula for the volume and the total scalar curvature of $\pi^{-1}(B(R)/\Gamma)$ for large $R$.

\begin{prop}\label{volexpansion}
For $R \to +\infty$ one has 
\begin{equation*}
\int_{\pi^{-1}(B(R)/\Gamma)} \frac{\eta^m}{m!} = \frac{\pi^m}{m!|\Gamma|} R^{2m} - \frac{4\pi^m e}{(m-1)!|\Gamma|} R^2 - \frac{4\pi^m(c-2e^2R^{4-2m})}{(m-2)!|\Gamma|} + \int_X \frac{\xi^m}{m!} + O(R^{-1})
\end{equation*}
(Note that the term of order $R^{4-2m}$ is not infinitesimal just in dimension $m=2$).
\end{prop}
\begin{proof}
Since $\eta = \xi + dd^c \varphi$, the volume form of $\eta$ is given by
\begin{equation*}
\eta^m = \xi^m + d\left( \sum_{\ell=1}^m \binom{m}{\ell} \xi^{m-\ell} \wedge d^c \varphi \wedge (dd^c \varphi)^{\ell-1}\right).
\end{equation*}
Recall that $\xi$ is supported in $K$, which in turns is compactly contained in $\pi^{-1}(B(R)/\Gamma)$.
Therefore, integrating formula above and applying Stokes' theorem gives
\begin{equation}\label{splitvolume}
\int_{\pi^{-1}(B(R)/\Gamma)} \eta^m = \int_X \xi^m + \int_{\partial B(R)/\Gamma} d^c \varphi \wedge (dd^c \varphi)^{m-1}.
\end{equation}
Equation \eqref{derivphiasym} yields $d^c \varphi = d^c f + O(|z|^{-2m-1})$ and $dd^c \varphi = dd^c f + O(|z|^{-2m-2})$, whence by easy calculations one gets
\begin{equation*}
d^c \varphi \wedge (dd^c \varphi)^{m-1} = (f')^m d^c|z|^2 \wedge (dd^c|z|^2)^{m-1} + O(|z|^{-2m-1}).
\end{equation*}
Integrating over $\partial B(R)/\Gamma$ and applying lemma \ref{lemareasfera} gives
\begin{equation}\label{intphiformboundary}
\int_{\partial B(R)/\Gamma} d^c \varphi \wedge (dd^c \varphi)^{m-1}
= \frac{(4\pi)^m R^{2m}f'(R^2)^m}{|\Gamma|} +O(R^{-1}).
\end{equation}
By \eqref{deff} one calculates 
$f'(t)^m = 4^{-m}\left(1 - 4me t^{1-m} - 4m(m-1)(ct^{-m}-2e^2t^{2-2m})\right) + O(t^{1-2m})$, whence the thesis follows after substituting in \eqref{intphiformboundary} and the result in \eqref{splitvolume}.
\end{proof}

We now aim to determine an asymptotic formula for the total scalar curvature of $\pi^{-1}(B(R)/\Gamma)$ as $R$ growths.
Our main interest is in Corollary \ref{totalscal}, which also follows by \cite[Theorem C]{HL} and the classical fact (see for example \cite{LeBrun1988}) that $e$ is, up to a positive normalization constant depending on the dimension $m$ and the order of $\Gamma$, the ADM mass of the ALE metric associated with $\eta$.
Nevertheless we include a complete, direct proof for the reader convenience.

\begin{prop}\label{prop::totalscalarsurvature}
Let $s$ be the scalar curvature of $\eta$.
For $R \to +\infty$ one has 
\begin{equation*}
\int_{\pi^{-1}(B(R)/\Gamma)} s \frac{\eta^m}{m!} = \int_X \frac{\rho \wedge \xi^{m-1}}{(m-1)!} + \frac{16\pi^me}{(m-2)!|\Gamma|} + O(R^{-2}).
\end{equation*}
\end{prop}
\begin{proof}
The Ricci form $\rho$ of $\eta$ satisfies $s \eta^m = m \rho \wedge \eta^{m-1}$.
Thanks to \eqref{eta=epsilon+ddcphi} one has
\begin{equation*}
\eta^{m-1} = \xi^{m-1} + d\left( \sum_{\ell=1}^{m-1} \binom{m-1}{\ell} \xi^{m-1-\ell} \wedge d^c \varphi \wedge (dd^c \varphi)^{\ell-1}\right).
\end{equation*}
Since $\xi$ is supported in $K$, which in turns is compactly contained in $\pi^{-1}(B(R)/\Gamma)$, integrating $m \rho \wedge \eta^{m-1}$ and applying Stokes' theorem gives
\begin{equation}\label{totalscalsplit1}
\int_{\pi^{-1}(B(R)/\Gamma)} s\eta^m = \int_X m \rho \wedge \xi^{m-1} + \int_{\partial B(R)/\Gamma} m \rho \wedge d^c \varphi \wedge (dd^c \varphi)^{m-2}.
\end{equation}
The fact that $\xi$ is supported in $K$, also implies that $\eta = dd^c \varphi$ in a neighborhood $U$ of $\partial B(R)/\Gamma$. 
On the other hand the two-form $\tilde \eta = dd^c f(|z|^2)$ defines a K\"ahler metric on $U$ at least if $R$ is sufficiently large.
Of course how large has to be $R$ depends on the value of the constants $e$ and $c$.
However, we are interested in large $R$ asymptotic, therefore we can suppose that $\tilde \eta$ is K\"ahler with no loss of generality.
Let $\tilde \rho$ be the Ricci form of $\tilde \eta$, so that on $U$ one has
\begin{equation*}
\rho = \tilde \rho - dd^c \log ((dd^c \varphi)^m / (dd^c f)^m).
\end{equation*}
By \eqref{derivphiasym} derivatives of $\varphi$ equals derivatives of $f(|z|^2)$ up to a controlled error.
In particular one has 
\begin{equation*}
dd^c \log ((dd^c \varphi)^m / (dd^c f)^m)
= O(|z|^{-2m-4}).
\end{equation*}
On the other hand, since $\tilde \eta = f'dd^c|z|^2 + f''d|z|^2 \wedge d^c|z|^2$ one can readily calculate
\begin{equation*}
\tilde \rho = - dd^c \log \left[ (f')^{m-1}\left( f'+ |z|^2f''\right)\right].
\end{equation*}
The second summand of \eqref{totalscalsplit1} is then given by
\begin{multline*}
\int_{\partial B(R)/\Gamma} m \rho \wedge d^c \varphi \wedge (dd^c \varphi)^{m-2} \\
= \int_{\partial B(R)/\Gamma} m d^c \log \left[ (f')^{m-1}\left( f'+ |z|^2f''\right)\right] \wedge (dd^c \varphi)^{m-1} + O(R^{-4}).
\end{multline*}
Again by \eqref{derivphiasym} one has $dd^c \varphi = dd^c f + O(|z|^{-2m-2})$ so that
\begin{multline}\label{ndclog...}
\int_{\partial B(R)/\Gamma} m \rho \wedge d^c \varphi \wedge (dd^c \varphi)^{m-2} \\
= \int_{\partial B(R)/\Gamma} m d^c \log \left[ (f')^{m-1}\left( f'+ |z|^2f''\right)\right] \wedge (dd^c f)^{m-1} + O(R^{-2}).
\end{multline}
Since $f$ depends only on $|z|^2$, one has 
\begin{equation}\label{defh}
d^c \log \left[ (f')^{m-1}\left( f'+ |z|^2f''\right)\right] \wedge (dd^c f)^{m-1} = h(|z|^2)
d^c|z|^2 \wedge (dd^c|z|^2)^{m-1}
\end{equation}
where $h = \left[ (f')^{m-1}\left( f'+ tf''\right)\right]' / \left( f'+ tf''\right)$.
By definition of $f$ and elementary computations one gets
\begin{equation*}
f'(t)^{m-1} = 4^{1-m}\left(1 - 4(m-1)e t^{1-m} - 4(m-1)^2 c t^{-m}\right) + O(t^{2-2m})
\end{equation*}
where the error term vanishes in dimension $m=2$.
Moreover one calculates
\begin{equation}\label{denominatore}
f'(t)+ tf''(t) = \frac{1}{4} + (m-2)e t^{1-m} + (m-1)^2 c t^{-m},
\end{equation}
so that
\begin{equation*}
f'(t)^{m-1}(f'(t)+ tf''(t)) = 4^{-m}(1 - 4e t^{1-m}) + O(t^{2-2m}),
\end{equation*}
and deriving yields
\begin{equation}\label{numeratore}
[f'(t)^{m-1}(f'(t)+ tf''(t))]' = 4^{1-m}(m-1)e t^{-m} + O(t^{1-2m}).
\end{equation}
Dividing \eqref{numeratore} by \eqref{denominatore} one finally gets the following expansion
\begin{equation*}
h(t) = 4^{2-m}(m-1)e t^{-m} + O(t^{1-2m}).
\end{equation*}
Therefore substituting \eqref{defh} into \eqref{ndclog...} and applying lemma \ref{lemareasfera} yields
\begin{equation*}
\int_{\partial B(R)/\Gamma} m \rho \wedge d^c \varphi \wedge (dd^c \varphi)^{m-2} \\
= \frac{m(m-1)(4\pi)^me}{|\Gamma|} + O(R^{-2}).
\end{equation*}
whence the thesis follows after substituting in \eqref{totalscalsplit1}.
\end{proof}

A straightforward consequence of Proposition \ref{prop::totalscalarsurvature} is that the scalar curvature is integrable on $X$.
More specifically it holds the following 
\begin{cor}[{\cite[Theorem C]{HL}}]\label{totalscal}
The total scalar curvature of $\eta$ is given by 
\begin{equation*}
\int_X s \frac{\eta^m}{m!} = \int_X \frac{\rho \wedge \xi^{m-1}}{(m-1)!} + \frac{16\pi^me}{(m-2)!|\Gamma|}.
\end{equation*}
\end{cor}

Among the consequences of Proposition \ref{volexpansion} and Corollary \ref{totalscal} there is the possibility of expressing cohomological quantities on $X$ like $\int_X \xi^m/m!$ and $\int_X \rho \wedge \xi^{m-1}/(m-1)!$ in terms of geometric quantities of the K\"ahler metric $\eta$ like the total scalar curvature and the volume growth of balls, together with the constants $e,c$ in the expansion of the K\"ahler potential at infinity.

Note that mentioned cohomological quantities appear in the formula \eqref{expansionfutaki} for the Futaki invariant of an orbifold resolution.
Thus they can be recovered by \emph{any} ALE K\"ahler metric $\eta_p = \xi_p + dd^c \varphi_p$ representing the fixed K\"ahler class on the model resolution $X_p$ of the singular point $p$. 
In order to get an explicit expression, write $\varphi_p$ as
\begin{equation*}
\varphi_p = \frac{1}{4}|z|^2 + e_p \frac{1-|z|^{4-2m}}{2-m} + c_p |z|^{2-2m} + O(|z|^{-2m}),
\end{equation*}
for some real constants $e_p$ and $c_p$ as $|z| \to +\infty$.
Thanks to Corollary \ref{totalscal} then one has
\begin{equation*}
\int_{X_p} \frac{\rho_p \wedge \xi_p^{m-1}}{(m-1)!} 
= \int_{X_p} s_p \frac{\eta_p^m}{m!} - \frac{16\pi^m e_p}{(m-2)!|\Gamma_p|},
\end{equation*}
where $s_p$ is the scalar curvature of $\eta_p$.
Moreover, proposition \ref{volexpansion} yields
\begin{multline}\label{eq::xi^mandc_p}
\int_{X_p} \frac{\xi_p^m}{m!}
= \frac{4\pi^mc_p}{(m-2)!|\Gamma_p|} \\
+ \lim_{R \to +\infty} \int_{\pi_p^{-1}(B(R)/\Gamma_p)} \frac{\eta_p^m}{m!} - \frac{\pi^m}{m!|\Gamma_p|} \left(R^{2m} - 4me_p R^2 + 8m(m-1)e_p^2 R^{4-2m}\right).
\end{multline}
Note that formula above reduces drastically whenever $\eta_p$ is Ricci-flat. In this case, the volume form of $\eta_p$ equals the euclidean volume form \cite[Proposition 2.4]{ADVLM} and moreover $e_p=0$, hence
\begin{equation}\label{eq::cpRicciflat}
\int_{X_p} \frac{\xi_p^m}{m!} = \frac{4\pi^mc_p}{(m-2)!|\Gamma_p|}.
\end{equation}

\section{Constant scalar curvature K\"ahler resolutions}

\label{sec:csck}

In this section we prove the existence and non-existence results for constant scalar curvature K\"ahler metrics on resolutions of K\"ahler orbifolds under suitable hypotheses.

The set-up is similar to that of the previous sections.
In particular, we assume that $\pi : M' \to M$ is a resolution of a $m$-dimensional orbifold $M$ having finite singular set $S \subset M$.
More precisely, around each singular point $p \in S$, the map $\pi$ is equal to a resolution $\pi_p : X_p \to \mathbf C^m / \Gamma_p$ restricted to some neighborhood of the exceptional set $\pi_p^{-1}(0)$.
Moreover we assume given a K\"ahler metric $\omega$ on $M$ with constant scalar curvature $s$ and we denote by $\mu : M \to \mathfrak g^*$ a moment map for the action of the group of holomorphic Hamiltonian diffeomorphisms of $(M,\omega)$, normalized so that $\int_M \mu \omega^n = 0$.
We also assume given for each singular point $p \in S$ a scalar-flat ALE metric $\eta_p$ on the model resolution $X_p$ of the form $\eta_p = \xi_p + dd^c \varphi_p$, with $\xi_p$ compactly supported around the exceptional set of $\pi_p$. 

Note that by Corollary \ref{totalscal} and Proposition \ref{prop::totalscalarsurvature} the scalar-flatness hypothesis of $\eta_p$ affects the expansion of the ALE K\"ahler potential $\varphi_p$ for large $z$.
More precisely, it gives a cohomological formula for the ADM mass of $\eta_p$.
In particular one has 
\begin{equation*}
\varphi_p =
\frac{1}{4}|z|^2 + e_p \frac{1-|z|^{4-2m}}{2-m} + c_p |z|^{2-2m} + O(|z|^{-2m})
\end{equation*}
with ADM mass $e_p = - \frac{|\Gamma_p|}{16\pi^m(m-1)} \int_{X_p} \rho_p \wedge \xi_p^{m-1}$.
Finally, consider for each $p \in S$ the positive constant $a_p = \frac{1}{16\pi^mm(m-1)}\int_{X_p} \xi_p^m$.
Note that $a_p$ is related to $c_p$ and $e_p$ by formula \eqref{eq::xi^mandc_p} and that, by discussion after that formula, one has $a_p=\frac{c_p}{4|\Gamma_p|}$ whenever $\eta_p$ is Ricci-flat.

Recalling that the Futaki invariant is an obstruction for the existence of constant scalar curvature K\"ahler metrics, the asymptotic formula for the Futaki invariant of Theorem \ref{thm::mainthmFutaki} and discussion above on the ADM mass readily give the following non-existence result

\begin{thm}
\label{nonex}
If one of the following conditions holds
\begin{itemize}
\item $\sum_{p \in S} \frac{e_p}{|\Gamma_p|}\mu(p) \neq 0$,
\item $\sum_{p \in S} a_p (s\mu + \Delta\mu)(p) \neq 0$,
\end{itemize}
then for all $\varepsilon>0$ sufficiently small the K\"ahler class
$\pi^*[\omega] + \varepsilon \sum_{p \in S} [\xi_p] \in H^{1,1}(M')$
contains no constant scalar curvature K\"ahler metrics.
\end{thm}

The starting point for our existence results is the following theorem on existence of extremal K\"ahler metrics on resolutions \cite{ArezzoLenaMazzieri2015}:
\begin{thm}\label{thmAreLenMazextremals}
With the notation above, assume that $\omega$ is an extremal K\"ahler metric on the orbifold $M$. Then for all $\varepsilon>0$ sufficiently small there is an extremal K\"ahler metric on $M'$ which in the K\"ahler class
$\pi^*[\omega] + \varepsilon \sum_{p \in S} [\xi_p] \in H^{1,1}(M')$.
\end{thm}

At this point we are in position to state and prove our main existence results.

\begin{thm}\label{somezero}
Let $e_p$ be the ADM mass of $\eta_p$, and let $Q \subset S$ be the subset of singular points with non-zero ADM mass.
If $Q$ is non-empty and
\begin{equation*}
\sum_{q \in Q} \frac{e_q}{|\Gamma_q|}\mu(q) = 0 \qquad \mbox{and} \qquad
\linspan\{\mu(q) \,|\, q \in Q\} = \mathfrak g^*
\end{equation*}
then, after identifying each $\xi_p$ with a $(1,1)$-form on $M'$ as above, for all $\varepsilon>0$ sufficiently small there exists $\lambda_q(\varepsilon) >0$ and a constant scalar curvature K\"ahler metric $\omega'_\varepsilon$ such that 
\begin{equation}\label{eq::Kcalsssomenonzero}
[\omega'_\varepsilon] = \pi^* [\omega] + \sum_{q \in Q} \lambda_q(\varepsilon) [\xi_q] + \varepsilon \sum_{p \in S \setminus Q} [\xi_p] \in H^{1,1}(M')
\end{equation}
and $\lambda_q(\varepsilon) \sim \varepsilon$ as $\varepsilon$ tends to $0$.
\end{thm}

An analytic proof of this result has been given in \cite{ArezzoLenaMazzieri2015}. We skip the new proof being similar and simpler to the one of the next result.

\begin{thm}\label{allzero}
Suppose that $e_p = 0$ for all $p \in S$.
If
\begin{equation}\label{eq::balancingandfullrank}
\sum_{p \in S} a_p(s \mu + \Delta \mu)(p) = 0 \qquad \mbox{and} \qquad
\linspan\{(s \mu + \Delta \mu)(p) \,|\, p \in S\} = \mathfrak g^*
\end{equation}
then, after identifying each $\xi_p$ with a $(1,1)$-form on $M'$ as above, for all $\varepsilon>0$ sufficiently small there exists $\lambda_p(\varepsilon) >0$ and a constant scalar curvature K\"ahler metric $\omega'_\varepsilon$ such that 
\begin{equation}\label{eq::Kcalssallzero}
[\omega'_\varepsilon] = \pi^* [\omega] + \sum_{p \in S} \lambda_p(\varepsilon) [\xi_p] \in H^{1,1}(M')
\end{equation}
and $\lambda_p(\varepsilon) \sim \varepsilon$ as $\varepsilon$ tends to $0$.
\end{thm}

This result extends an analogue one in \cite{ADVLM}, proved under the additional assumption of Ricci-flatness of the local resolutions, in which case $a_p=\frac{c_p}{4|\Gamma_p|}$, as remarked above.

\begin{proof}
For each $p \in S$ and for all real $t_p$ such that $|t_p|<1$, take $\varepsilon>0$ sufficiently small and consider on $M'$ the K\"ahler metric
\begin{equation*}
\omega_{t,\varepsilon} 
= \pi^* \omega + \varepsilon \sum_{p \in S} (1+t_p)^{1/m} \xi_p.
\end{equation*}
Note that this metric is invariant with respect to any holomorphic vector field with zeroes $V$ on $M'$ since $\pi^*\omega$ and each $\xi_p$ are.
Moreover note that any such vector field is Hamiltonian with respect to $\omega_{t,\varepsilon}$.
In particular, by theorem \ref{thm::mainthmFutaki}, one has
\begin{equation}\label{eq::Futomegatepsilon}
\Fut(V,\omega_{t,\varepsilon})
= - \frac{16\pi^m\varepsilon^m}{(m-2)!}  \sum_{p \in S} (1+t_p) a_p \left(s (u(p) - \underline{u}) + \Delta u(p)\right) + O(\varepsilon^{m+1}),
\end{equation}
where we used the hypothesis that $\omega$ has constant scalar curvature (hence vanishing Futaki invariant), for each $p\in S$ the model metric $\eta_p$ is Ricci flat (hence $\rho_p=0$), and the scalar curvature $s$ of $\omega$ is constant.
On the other hand, note that by general theory of Futaki invariant, $\Fut(V,\omega_{t,\varepsilon})$ depends polynomially on the cohomology class of $\omega_{t,\varepsilon}$, hence on $\varepsilon$ and $(1+t_p)^{1/m}$.
Finally, observe that the normalized Hamiltonian potential $u - \underline u$ of $\pi_*V$ is equal to $\langle \mu,V \rangle$.
Therefore, letting
\begin{equation*}
F(t,\varepsilon) (V) = - \frac{(m-2)!}{16\pi^m\varepsilon^m} \Fut(V,\omega_{t,\varepsilon})
\end{equation*}
defines a smooth function $F$ on $(-1,1)^{|S|} \times \mathbf R$ with values in $\mathfrak g^*$, being $|S|$ the cardinality of the singular set $S$.
By \eqref{eq::Futomegatepsilon}, for small $\varepsilon$ one has
\begin{equation*}
F(t,\varepsilon)
= \sum_{p \in S} (1+t_p) a_p \left(s \mu + \Delta \mu\right)(p) + O(\varepsilon).
\end{equation*}
Therefore, hypotheses \eqref{eq::balancingandfullrank} ensure that $F(0,0)=0$, and the Jacobian $\partial F/ \partial t$ at the point $(0,0)$ has rank equal to $\dim \mathfrak g$.
As a consequence, by implicit function theorem one can find $\varepsilon_0>0$ and a smooth function
$t$ on $(-\varepsilon_0,\varepsilon_0)$ with values to $(-1,1)^{|S|}$ such that $F(t(\varepsilon),\varepsilon) = 0$, and $t(0)=0$.
Clearly there are $|S|-\dim \mathfrak g$ free parameters in doing this, but we don't need this extra flexibility for our purposes.
By discussion above, for all holomorphic vector field with zeroes on $M'$ then one has $\Fut(V,\omega_{t(\varepsilon),\varepsilon})=0$ if $0 < \varepsilon < \varepsilon_0$. 
Therefore, letting $\lambda_p(\varepsilon) = \varepsilon (1+t(\varepsilon)_p)^{1/m}$ yields a family of K\"ahler classes
\begin{equation}\label{eq::classesFut=0}
\pi^* [\omega] + \sum_{p \in S} \lambda_p(\varepsilon) [\xi_p] \in H^{1,1}(M')
\end{equation}
with vanishing Futaki invariant and approaching the class $[\pi^*\omega]$ as $\varepsilon \to 0$.

By the elementary remark that extremal K\"ahler metrics with vanishing Futaki invariant have constant scalar curvature \cite{Calabi1985, Futaki1988}, in order to get the thesis we are now reduced to show that classes as in \eqref{eq::classesFut=0} contain an extremal K\"ahler metric, at least when $\varepsilon$ is sufficiently small.
This follows by Theorem \ref{thmAreLenMazextremals} and openness of the extremal cone \cite{LeBrunSimanca1993}.
Indeed, by a standard perturbation argument these two theorems imply that under our hypotheses there is a small open ball $H^{1,1}(M')$ centered at $[\pi^*\omega]$ whose intersection $A$ with the K\"ahler cone of $M'$ is constituted by extremal K\"ahler classes, i.e. K\"ahler classes representable by extremal K\"ahler metrics.
Since classes as in \eqref{eq::classesFut=0} are contained in $A$ for $\varepsilon$ sufficiently small, it follows that all these classes contain extremal K\"ahler metrics.
\end{proof}

Comparing formulae for K\"ahler classes \eqref{eq::Kcalsssomenonzero} and \eqref{eq::Kcalssallzero} somehow suggests that one can still get cscK metrics in adiabatic classes of $M'$ in cases not covered by theorems above.
This can be done by choosing different scaling volumes with respect to $\varepsilon$ to excepional divisors, according they project via $\pi$ to a singular point with zero or non-zero ADM mass.
More precisely one has the following result which is not covered by previous works.
\begin{thm}\label{diffvolumes}
Let $Q \subset S$ be the subset constituted by those singular points of $M$ which have non-zero ADM mass and let $P$ its complement.
If $P$ and $Q$ are both not empty and 
\begin{equation*}
\sum_{q \in Q} \frac{e_q}{|\Gamma_q|}\mu(q) + \sum_{p \in P} a_p(s \mu + \Delta \mu)(p) = 0 
\quad \mbox{and} \quad
\linspan\{\mu(q), (s \mu + \Delta \mu)(p)  \,|\, q \in Q, \, p\in P\} = \mathfrak g^*
\end{equation*}
then, after identifying each $\xi_p$ with a $(1,1)$-form on $M'$ as above, for all $\varepsilon>0$ sufficiently small there exist $\lambda_p(\varepsilon)>0$ and a constant scalar curvature K\"ahler metric $\omega'_\varepsilon$ such that 
\begin{equation*}
[\omega'_\varepsilon] = \pi^* [\omega] + \sum_{q \in Q} \lambda_q(\varepsilon) [\xi_q] + \sum_{p \in P} \lambda_p(\varepsilon) [\xi_p] \in H^{1,1}(M').
\end{equation*}
Moreover, as $\varepsilon$ tends to $0$, one has $\lambda_q(\varepsilon) \sim \varepsilon$ for all $q \in Q$ and $\lambda_p(\varepsilon) \sim \varepsilon^\frac{m-1}{m}$ for all $p \in P$.
\end{thm}
\begin{proof}
The statement follows exactly from the same line of arguments as in the proof of theorem \ref{allzero} once one starts with the K\"ahler metric on $M'$ defined by
\begin{equation*}
\omega_{t,\varepsilon} 
= \pi^* \omega + \varepsilon \sum_{q \in Q} (1+t_q)^\frac{1}{m-1} \xi_q + \varepsilon^\frac{m-1}{m} \sum_{p \in P} (1+t_p)^\frac{1}{m} \xi_p
\end{equation*}
with $t_p,t_q \in \mathbf R$ such that $|t_p|,|t_q|<1$, and $\varepsilon>0$ sufficiently small.
\end{proof}


\begin{thebibliography}{99}

\bibitem{ADVLM} C. Arezzo, A. Della Vedova, R. Lena and L. Mazzieri.
		\emph{``On the Kummer construction for KCSC metrics"}. 
		To appear in the special volume of Bollettino Unione Matematica Italiana, in memory of Paolo De Bartolomeis.

\bibitem{ArezzoLenaMazzieri2015} C. Arezzo, R. Lena and L. Mazzieri.
		\emph{``On the resolution of extremal and constant scalar curvature K\"ahler orbifold''}.
		Int. Math. Res. Not. IMRN 2016, no. 21, 6415--6452. 

\bibitem{AP} C. Arezzo and F. Pacard.
		\emph{``Blowing up and desingularizing constant scalar curvature K\"ahler manifolds''}. 
		Acta Math. \textbf{196} (2006), no. 2, 179--228.
  
\bibitem{AP2} C. Arezzo and F. Pacard.
		\emph{``Blowing up K\"ahler manifolds with constant scalar curvature II''}. 
		Ann. of Math. (2) \textbf{170} (2009), no. 2, 685--738. 

\bibitem{Calabi1985} E. Calabi.
		\emph{``Extremal K\"ahler metrics, II''}.
		In Differential Geometry and Complex Analysis, 
		eds. I. Chavel and H. M. Farkas, Springer Verlag (1985), 95--114.
		
\bibitem{DVZ}  A. Della Vedova and F.  Zuddas.  \emph{`` Scalar curvature and asymptotic chow stability of projective bundles and blowups"}. Trans. Amer. Math. Soc., \textbf{364}, (2012), 6495-6511.
		
\bibitem{DT} W. Ding and G. Tian.
		\emph{``K\"ahler-Einstein metrics and the generalized Futaki invariants''}.
		Invent. Math., \textbf{110} (1992), 315--335.
		
\bibitem{Do} S.K. Donaldson.
		\emph{``Scalar curvature and stability of toric varieties''}.
		J. Differential Geom. \textbf{62} (2002), no. 2, 289--349.

\bibitem{Futaki1988} A. Futaki.
		\emph{``K\"ahler-Einstein metrics and Integral Invariants''}.
		Springer-LNM 1314, Springer-Verlag (1988).
		
\bibitem{HL} H.-J. Hein and C. LeBrun.
		\emph{``Mass in K\"ahler geometry''}. 
		Comm. Math. Physics \textbf{347} (2016), no.~1, 183--221. 

\bibitem{Joyce} D. D. Joyce.
		\emph{``Compact manifolds with special holonomy''}. 
		Oxford Mathematical Monographs. Oxford University Press, Oxford, 2000.
		
\bibitem{LeBrun1988} C. LeBrun.
		\emph{``Counter-Examples to the Generalized Positive Action Conjecture''}.
		 Comm. Math. Phys. \textbf{118} (1988), 591--596.		
		
\bibitem{LeBrunSimanca1993} C. LeBrun and S. R. Simanca.
		\emph{``On the K\"ahler classes of extremal metrics''}.
		Geometry and global analysis (Sendai, 1993),  255--271, Tohoku Univ., Sendai, 1993.

\bibitem{Od} Y. Odaka. 
		\emph{``The GIT stability of polarized varieties via discrepancy''}.
		Ann. of Math. \textbf{177} (2013), no. 2, 645--661.
  
\bibitem{St} J. Stoppa.
		\emph{``Unstable blowups"  }. J. Algebraic Geom. \textbf{19} (2010) 1-17.

\bibitem{Gabor1} G. Sz{\'e}kelyhidi.
		\emph{``On blowing up extremal K\"ahler manifolds''}.
		Duke Math. J. \textbf{161} (2012), no.~8, 1411--1453. 
  
\bibitem{Gabor2} G. Sz{\'e}kelyhidi.
		\emph{``On blowing up extremal K\"ahler manifolds II''}.
		Invent. Math. \textbf{200} (2015), no. 3, 925--977.
		
\bibitem{Tian} G. Tian.
		\emph{``Extremal metrics and Geometric stability''}.
		Houston Journ. of Math., {\bf{28}}, no.2, 2002, 411--431.

	
\end{thebibliography}
\end{document}